\documentclass[12pt]{amsart}
\usepackage{amsmath,amstext, amsfonts,amssymb}
\usepackage{mathrsfs}
\usepackage[T2A]{fontenc}
\usepackage[utf8]{inputenc}
\usepackage[english]{babel}
\usepackage{cite, color}

\usepackage{pgfplots}
\pgfplotsset{compat=1.18}
\usepgfplotslibrary{groupplots}

\newcommand{\RR}{\mathbb R}
\newcommand{\NN}{\mathbb N}

\newcommand{\kk}{{\bf k}}
\newcommand{\kUpper}{{\kk^\uparrow}}

\newtheorem{theorem}{Theorem}
\newtheorem{lemma}{Lemma}

\allowdisplaybreaks

\begin{document}
\title[Kolmogorov-type inequalities for more than three norms]{A note on the Kolmogorov-type inequalities for more than three norms}
\author[O.~V.~Kovalenko]{Oleg Kovalenko}
\address{Department of Mathematical Analysis and Optimization, Oles Honchar Dnipro National University, Dnipro, Ukraine}
\email{olegkovalenko90@gmail.com}

\begin{abstract}
 In this note we show that sharp Kolmogorov-type inequalities that  estimate the uniform norm $\|f^{(k)}\|$ of the $k$-th derivative of a function $f\colon \RR\to\RR$  by the values of the uniform norm of $f$ and uniform norms of several its higher derivatives ($\|f^{(r)}\|$ and $\|f^{(r-1)}\|$, or $\|f^{(r)}\|$ and $\|f^{(r-2)}\|$, or $\|f^{(r)}\|$, $\|f^{(r-1)}\|$ and $\|f^{(r-2)}\|$) using standard techniques can be obtained from the known solutions to the Kolmogorov problem about existence of a function with given norms of its derivatives. 
\end{abstract}

\maketitle
\section{Introduction}
Inequalities for derivatives is an important topic in approximation theory that has many applications. First sharp (i.e., with the smallest constant) inequality was obtained by Landau~\cite{Landau13b} in 1913. He proved that if $a,b>0$ and a function $f\colon \RR\to\RR$ is such that $|f(t)|\leq a$ and $|f''(t)|\leq b$ for all $t\in \RR$, then 
\begin{equation}\label{LandauInequality}
  |f'(t)|\leq \sqrt{2ab}, t\in\RR,
\end{equation}
 and the constant $\sqrt{2}$ can not be decreased.

Denote by $L_\infty$ the class of measurable essentially bounded functions $f\colon \RR\to\RR$ and let $\|\cdot\|$ be the uniform norm. For $r\in\NN$ by  $L^r_\infty$ we denote the class of continuous functions $f\colon \RR\to\RR$ that have $r-1$ derivatives, $f^{(r-1)}$ is locally absolutely continuous, and $f^{(r)}\in L_\infty$. Set $L^r_{\infty,\infty}:= L_\infty\cap L^r_\infty$. 

Denote by $\varphi_r$, $r\in\NN$, the Euler perfect spline i.e., the $r$-th periodic primitive  of the function ${\rm sgn}\sin x$ with zero mean value; for $\lambda > 0$ set $\varphi_{\lambda, r}(t) :=\lambda^{-r}\varphi(\lambda t)$, $t\in\RR$. 

In 1938 Kolmogorov~\cite{Kolmogorov1938} proved that if $0<k<r$ and $f\in L^r_{\infty,\infty}$, then $f^{(k)}\in L_\infty$ and 
\begin{equation}\label{kolmogorovInequality}
\|f^{(k)}\|\leq \frac{\|\varphi_{r-k}\|}{\|\varphi_r\|^{1-\frac kr}} \|f\|^{1-\frac kr}\|f^{(r)}\|^{\frac kr},    
\end{equation}
and the constant $\frac{\|\varphi_{r-k}\|}{\|\varphi_r\|^{1-\frac kr}}$ can not be reduced. Such inequalities for derivatives are often called the Kolmgorov-type or Landau--Kolmogorov-type inequalities. These inequalities were heavily studied, we refer to monograph~\cite{BKKP}, survey~\cite{Arestov1996} and references therein  for an overview of many results in this area.

In the proof of inequality~\eqref{kolmogorovInequality} Kolmogorov proved the comparison theorem, which is arguably every bit as important as the inequality itself.  For $f\in L^r_{\infty,\infty}$ and $0\leq k\leq r$ set $M_k(f):= \|f^{(k)}\|$.
\begin{theorem}\label{th::kolmogorovComparisonTheorem}
    Let $f\in L^r_{\infty,\infty}$ and $a,\lambda > 0$ be such that $M_0(f) \leq M_0(a\varphi_{\lambda,r})$ and $M_r(f) \leq M_r(a\varphi_{\lambda,r})$. If $\xi,\eta\in\RR$ are such that $f(\xi) = a\varphi_{\lambda,r}(\eta)$, then 
    $$
   |f'(\xi)|\leq |a\varphi'_{\lambda,r}(\eta)|.
    $$
\end{theorem}

Let $X\subset L^r_{\infty,\infty}$ be some non-empty set and $0<k<r$.
The quantity 
\begin{equation}\label{modulusOfContinuity}
  \omega(D^k, X;\delta) = \omega(X;\delta):= \sup_{f\in X, M_0(f)\leq \delta}M_k(f), \delta >0  
\end{equation}
is called the modulus of continuity of the differential  operator $D^k$ on the class $X$. We refer to~\cite[Chapter~7]{BKKP} for a discussion of the connection between the Kolmogorov-type inequalities and related problems, including the problem of finding the modulus of continuity of an operator. Observe that the fact that inequality~\eqref{kolmogorovInequality} holds on $L^r_{\infty,\infty}$ and is sharp, is equivalent to equality
$$
\omega(D^k, W^r_{\infty,\infty};\delta) = \frac{\|\varphi_{r-k}\|}{\|\varphi_r\|^{1-\frac kr}} \delta^{1-\frac kr}, \delta > 0,
$$
where $W^r_{\infty,\infty} = \{f\in L^r_{\infty,\infty}\colon M_r(f)\leq 1\}$. 
If one has 
$$
X = \{f\in L^r_{\infty,\infty}\colon p(f)\leq 1\}
$$
with some non-negative functional $p$ such that $p(c\cdot f) = c^\alpha p(f)$ for all $f\in L^r_{\infty,\infty}$, all $c > 0$, and some $\alpha\in\RR$, then for all $f$ such that $p(f)\neq 0$, one has
\begin{equation}\label{inequalityInTermsOfModulusOfContinuity}
\|f^{(k)}\| =p^{\frac 1\alpha}(f)\left\|\frac{f^{(k)}}{p^{\frac 1\alpha}(f)}\right\|\leq p^{\frac 1\alpha}(f)\cdot \omega\left(D^k,X;\frac{\|f\|}{p^{\frac 1\alpha}(f)}\right).
\end{equation}
In particular, if $p(f) = 0\iff f = 0$ and we know the function $\omega$, then inequality~\eqref{inequalityInTermsOfModulusOfContinuity} holds for all $f\in L^r_{\infty,\infty}$ and can be viewed as an abstract Kolmogorov-type inequality that estimates the value $\|f^{(k)}\|$ in terms of the quantities $\|f\|$ and $p(f)$.

Another related problem, which was formulated by Kolmogorov in 1926 (we will refer to it as Kolmogorov's problem) is a follows. Assume that a set of integer numbers 
$$
0\leq k_1 < k_2<\ldots < k_d\leq r, {\bf k} = (k_1,\ldots, k_d)
$$
is given. The problem is to find necessary and sufficient conditions on positive numbers $M_{k_1},\ldots, M_{k_d}$, $M_{\bf k} = (M_{k_1},\ldots, M_{k_d})$  to guarantee existence of a function $f\in L_{\infty,\infty}^r$ such that 
$$
M_{\bf k}(f) = M_{\bf k}, \text{ where }M_{\bf k}(f) = (M_{k_1}(f),\ldots, M_{k_d}(f)).
$$
Kolmogorov's article~\cite{Kolmogorov1938} was actually devoted to a solution of this problem in the case $d = 3$, $k_1 = 0, k_2 = k, k_3 = r$. In this case the required necessary and sufficient condition is the inequality
$$
M_k\leq \frac{\|\varphi_{r-k}\|}{\|\varphi_r\|^{1-\frac kr}} M_0^{1-\frac kr}M_r^{\frac kr}.
$$
We refer to articles~\cite{Kovalenko15a,Kovalenko19} for an overview of known results regarding the Kolmogorov problem. A convenient alternative formulation of the Kolmogorov problem was suggested in~\cite{Kovalenko15a}. Given a fixed vector ${\bf k}$ (of orders of the derivatives as above), find a minimal (in some sense) family $F_{\bf k}\subset L^r_{\infty,\infty}$ such that 
\begin{equation}\label{alternativeKolmogorovProblem}
M_{\bf k}(L^r_{\infty,\infty}) = M_{\bf k}(F_{\bf k}),
\end{equation}
where for a set $X\subset L^r_{\infty,\infty}$ we use the notation 
$$M_{\bf k}(X) = \{M_{\bf k}(f)\colon f\in X\}.$$
Under these notations Kolmogorov's result can be formulated as follows. For ${\bf k} = (0, k, r)$, $0< k<r$, one has 
$$
F_{\bf k} = \{a\cdot\varphi_{\lambda,r} + b\colon a,\lambda > 0, b\geq 0\},
$$
and Theorem~\ref{th::kolmogorovComparisonTheorem} is the key result to prove that equality~\eqref{alternativeKolmogorovProblem} holds.

Recently the Kolmogorov-type inequalities that involve more than three norms were investigated. In particular, in~\cite{Dragomir} the following inequalities
\begin{equation}\label{dragomirInequality}
\|f'\|\leq C_\eta\cdot \|f\|^{\frac{1+\eta}{2+\eta}}\cdot \|f''\|^{\frac{1-\eta}{2+\eta}}\cdot \|f'''\|^{\frac{\eta}{2+\eta}},
\end{equation}
with explicitly found sharp constants $C_\eta$, $\eta\in [0,1]$, were proved. 

The goal of this note is to show that such kind of inequalities can be deduced from the known solutions for the Kolmogorov problem. Namely, we show that a sharp Kolmogorov-type inequality that estimates
$\|f^{(k)}\|$ via  $\|f\|$, $\|f^{(r-1)}\|$ and $\|f^{(r)}\|$, $0<k<r-1$, follows from the known  (see~\cite{Dzyadyk75}) solution of the Kolmogorov problem with ${\bf k} = (0,k,r-1,r)$; a sharp Kolmogorov-type inequality that estimates $\|f^{(k)}\|$ via  $\|f\|$, $\|f^{(r-2)}\|$ and $\|f^{(r)}\|$, $0<k<r-2$ follows from the known  (see~\cite{Babenko12}) solution of the Kolmogorov problem with ${\bf k} = (0,k,r-2,r)$; a sharp Kolmogorov-type inequality that estimates
$\|f^{(k)}\|$ via  $\|f\|$, $\|f^{(r-2)}\|$, $\|f^{(r-1)}\|$ and $\|f^{(r)}\|$, $0<k<r-2$, follows from the known  (see~\cite{Rodov46}) solution of the Kolmogorov problem with ${\bf k} = (0,k,r-2,r-1,r)$.
\section{Connection Between the Kolmogorov Problem and the Kolmogorov-Type Inequalities}
\subsection{Extremal Splines}
 First of all, we define a family of polynomial splines that are extremal for the Kolmogorov problem in the cases ${\bf k}\in K$, where
\begin{equation}\label{kVectors}
 K = \{(0,k,r-1,r), (0,k,r-2,r), (0,k,r-2,r-1,r)\};
\end{equation}
here $k,r$ are natural numbers such that $k<r-1$ for ${\bf k} = (0,k,r-1,r)$, and $k<r-2$ for other two elements of the set $K$.
For all $a,b,c\geq 0$ set 
$$
\psi_0(a,b,c;t) = \begin{cases}
    0, & t\in [0,a],\\
    1, & t\in (a,a+b],\\
    0, & t\in (a+b, a+b+c].
\end{cases}
$$
Continue $\psi_0(a,b,c)$ evenly with respect to the point $a+b+c$ and oddly with respect to the point $2(a+b+c)$; finally continue the function $\psi_0$ to the whole line periodically with period $4(a+b+c)$. Let $\psi_s = \psi_s(a,b,c)$ be the $s$-th periodic primitive of $\psi_0$ with zero mean value, $s\geq 1$. These functions were introduced by Rodov~\cite{Rodov46}. We note that 
$$
\psi_1(a,b,c;t) = 
\begin{cases}
    -b, & t\in [0,a],\\
    t - a -b, & t\in (a,a+b],\\
    0, & t\in (a+b, a+b+c],
\end{cases}
$$
is odd with respect to the point $a+b+c$ and even with respect to $2(a+b+c)$; 
\begin{equation}\label{psi2}
\psi_2(a,b,c;t) = 
\begin{cases}
    -bt, & t\in [0,a],\\
    \frac{(t - a -b)^2}2 -ab - \frac {b^2}{2}, & t\in (a,a+b],\\
    -ab - \frac {b^2}{2}, & t\in (a+b, a+b+c],
\end{cases}
\end{equation}
is even with respect to $a+b+c$ and odd with respect to $2(a+b+c)$.

One period of the functions $\psi_0, \psi_1$ and $\psi_2$ is drawn on the picture below.

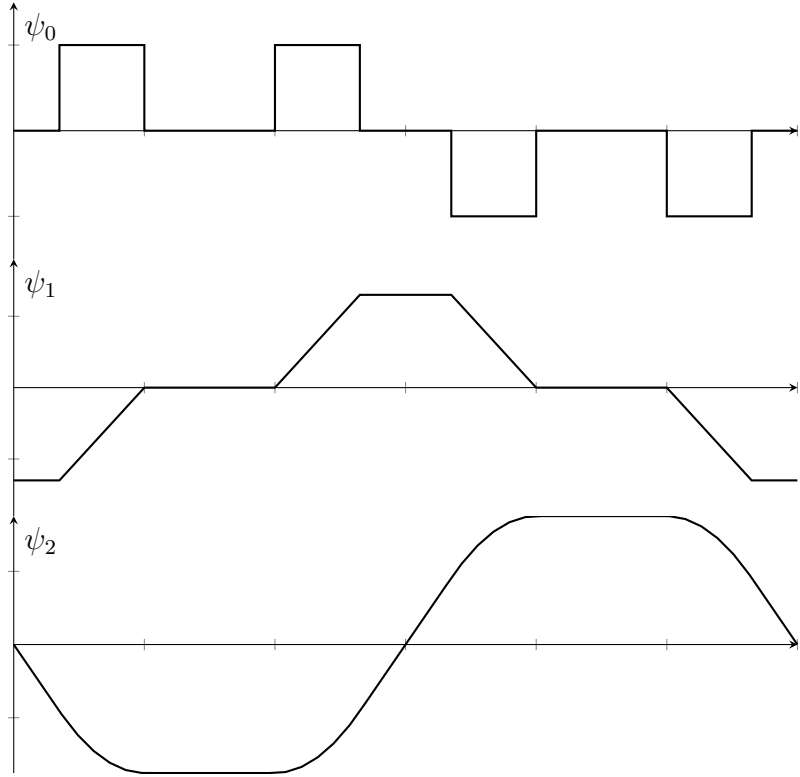
\begin{tikzpicture}
    \pgfmathsetmacro{\a}{0.7}
    \pgfmathsetmacro{\b}{1.3}
    \pgfmathsetmacro{\c}{1}
    \pgfmathsetmacro{\L}{\a + \b + \c}

    \begin{groupplot}[
        group style={
            group size=1 by 3,
            vertical sep=0cm,
            x descriptions at=edge bottom
        },
        width=12cm, height=5cm,
        grid=none,
        axis lines=middle,
        xmin=0, xmax={4*\L},
        no marks,
        yticklabels={},
        xticklabels={},
    ]
    \nextgroupplot[ylabel={$\psi_0$}, ymin=-1.5, ymax=1.5]
    \addplot[thick] coordinates {
        (0,0) (\a,0) (\a,1) (\a+\b,1) (\a+\b,0) (\L,0)          
        (\L+\c,0) (\L+\c,1) (2*\L-\a,1) (2*\L-\a,0) (2*\L,0)    
        (2*\L+\a,0) (2*\L+\a,-1) (2*\L+\a+\b,-1) (2*\L+\a+\b,0) (3*\L,0) 
        (3*\L+\c,0) (3*\L+\c,-1) (4*\L-\a,-1) (4*\L-\a,0) (4*\L,0) 
    };

    \nextgroupplot[ylabel={$\psi_1$}, ymin={-\b-0.5}, ymax={\b+0.5}]
    \addplot[thick] coordinates {
        (0,-\b) (\a,-\b) (\a+\b,0) (\L,0)                       
        (\L+\c,0) (2*\L-\a,\b) (2*\L,\b)                        
        (2*\L+\a,\b) (2*\L+\a+\b,0) (3*\L,0)                    
        (3*\L+\c,0) (4*\L-\a,-\b) (4*\L,-\b)                    
    };

    \nextgroupplot[ylabel={$\psi_2$},
        declare function={
            p2_base(\t) = ifthenelse(\t <= \a, -\b*\t, ifthenelse(\t <= \a+\b, 0.5*(\t-\a-\b)^2 - \a*\b - 0.5*\b^2, -\a*\b - 0.5*\b^2));
            p2_h(\t) = ifthenelse(\t <= \L, p2_base(\t), p2_base(2*\L-\t));
            p2_f(\t) = ifthenelse(\t <= 2*\L, p2_h(\t), -p2_h(\t-2*\L));
        }
    ]
    \addplot[thick, domain=0:4*\L, samples=50] {p2_f(x)};
    \end{groupplot}
\end{tikzpicture}

The formulae above imply that 
\begin{equation}\label{splineNorms}
\|\psi_0(a,b,c)\| = 1; \|\psi_1(a,b,c)\| = b; \|\psi_2(a,b,c)\| = ab + \frac{b^2}{2}.
\end{equation}
We also note that for $r\in\NN$ and $b >0$, $\psi_r(0,b,0)$ is the Euler's perfect spline $\varphi_{r,\lambda}$ with appropriately chosen parameter $\lambda$.

Define families of splines 
$$
S_{(0,k,r-2,r-1,r)} = \{\alpha\cdot\psi_r(a,b,c)\colon b,\alpha > 0, a,c\geq 0\},
$$
$$
S_{(0,k,r-1,r)} = \{\alpha\cdot\psi_r(a,b,0)\colon b,\alpha > 0, a\geq 0\},
$$
and 
$$
S_{(0,k,r-2,r)} = \{\alpha\cdot\psi_r(0,b,c)\colon b,\alpha > 0, c\geq 0\}.
$$
The solutions to the Kolmogorov problem from articles~\cite{Rodov46,Dzyadyk75,Babenko12} can now be stated in the alternative form as follows: for ${\bf k}\in K$
\begin{equation}\label{alternativeSolution}
F_{{\bf k}} = \{\psi+ d\colon \psi\in S_{\bf k}, d\geq 0\}.   
\end{equation}
Below we use the following notations:
$(0,k,r-2,r-1,r)^\uparrow = (r-2,r-1,r)$, $(0,k,r-1,r)^\uparrow = (r-1,r)$, $(0,k,r-2,r)^\uparrow = (r-2,r)$, so that for  a vector $\kk\in K$ the vector $\kUpper$ represents the vector of higher orders of derivatives. By $|\kk|$ we denote the number of coordinates of the vector $\kk$, so for each $\kk\in K$, one has $|\kk| = |\kUpper|+2$. With some abuse of notation we write $s\in \kk$ to mean that $s$ is some coordinate of the vector $\kk$, and $\min\kk$ to denote the minimal coordinate of $\kk$.

\begin{lemma}\label{l::splinesFamily}
    For each $\kk\in K$ and non-zero $f\in L^r_{\infty,\infty}$ one can construct a one-parametric family of splines $\psi_{\kk}(\beta)\in S_{\kk}$, $\beta\geq 0$ such that $M_{\kUpper}(f) = M_{\kUpper}(\psi_{\kk})$ for all $\beta$, and $M_s(\psi_{\kk}(\beta))$ continuously increases to $\infty$ as $\beta\to\infty$, $0\leq s < \min\kUpper$.
\end{lemma}
\begin{proof}
    Existence of the required families of splines follows from~\eqref{splineNorms}. For $\kk = (0,k,r-2,r)$ one can consider the family $\alpha \psi_r(0,b,\beta)$, $\beta\geq 0$, with $\alpha = M_r(f)$ and $b = \sqrt{\frac{2M_{r-2}(f)}{M_r(f)}}$.  Analogously in the case $\kk = (0,k,r-1,r)$ one can consider the family $\alpha \psi_r(\beta,b,0)$, $\beta\geq 0$ with $\alpha = M_r(f)$ and $b = \frac{M_{r-2}(f)}{M_r(f)}$. In the case $\kk = (0,k,r-2,r-1,r)$, the required family is $\alpha\psi_r(a,b,\beta)$, $\beta\geq 0$, where $\alpha = M_r(f)$, $b = \frac{M_{r-2}(f)}{M_r(f)}$ and $a = \frac{M_{r-2}}{M_{r-1}} - \frac{M_{r-1}}{2M_r}$. The fact that $a\geq 0$ follows from Landau's inequality~\eqref{LandauInequality}.

    The statement about the norm $M_s(\psi_{\kk}(\beta))$ follows from the definition of the splines.
\end{proof}

The following comparison theorem is a generalization of Theorem~\ref{th::kolmogorovComparisonTheorem}.
\begin{theorem}
    Let $f\in L^r_{\infty,\infty}$, $K$ be defined by~\eqref{kVectors}, ${\bf k}\in K$ and $\psi_{\bf k}\in S_{\bf k}$ be such that 
    $$
    M_s(f) \leq M_s(\psi_{\bf k}), s\in \kUpper\text{ or } s = 0.
    $$
    If $\xi,\eta\in\RR$ are such that $f(\xi) = \psi_{\bf k}(\eta)$, then 
    $$
   |f'(\xi)|\leq |\psi_{\bf k}'(\eta)|.
    $$
\end{theorem}
The proof of this theorem is contained in~\cite[Theorem~2]{Babenko12} for the case ${\bf k} = (0,k,r-2,r)$, and it can be proved using similar arguments for the other two cases. 

Again, this comparison theorem is arguable more important than the solutions of the Kolmogorov problem for the cases $\kk\in K$. A comparison theorem is not only a key result for the solutions of the Kolmogorov problem, but also allows to obtain solutions to various other extremal problems using standard arguments.  We refer to~\cite[Chapter~3]{ExactConstants} and~\cite[Chapter~2]{Babenko25OstrowskiBook}, where solutions to various extremal problems  on classes with a given comparison function are given.

\subsection{Existence of Extremal Splines}
The following result implies that for classes of functions that are defined in terms of some restrictions on the uniform norms of their higher derivatives, the modulus of continuity~\eqref{modulusOfContinuity} of the differential operator $D^k$ attains its values on the splines. It is essentially a direct consequence of solution~\eqref{alternativeSolution} of the Kolmogorov problem.
\begin{theorem}\label{th::main}
    Let $K$ be defined by~\eqref{kVectors}, $\kk\in K$ and 
    $$
    X = M_{\kUpper}^{-1}(A) = \{f\in L^r_{\infty,\infty}\colon M_{\kUpper}(f)\in A\},
    $$
    where $A\subset (0,\infty)^{|\kUpper|}$ is a non-empty set that satisfies the following property: if a sequence $\left\{a^n = \left(a^n_1,\ldots, a^n_{|\kUpper|}\right)\right\}\subset A$ is such that $a^n_{|\kUpper|}\to\infty$, then for some $1\leq i < |\kUpper|$ one has $a^n_i\to 0$, $n\to\infty$. If $\delta >0$ is such that $\{f\in X\colon M_0(f)\leq \delta\}\neq \emptyset$, then the supremum in the definition~\eqref{modulusOfContinuity} of the modulus of continuity is attained on some spline $\psi_{\kk}\in S_{\kk}\cap X$ such that $M_0(\psi_{\kk}) = \delta$.
\end{theorem}
\begin{proof}
    Assume that $f_1,\ldots, f_n,\ldots$ is a sequence of functions from $X$ such that $M_0(f_n) \leq \delta$ for all $n\in\NN$ and $\lim_{n\to\infty} M_k(f_n)= \omega(X;\delta)$. According to~\eqref{alternativeSolution} for each $n\in\NN$ we can choose a spline $\psi_{n} \in S_{\kk}$ and a number $d_n\geq 0$ such that $M_{\kk}(f) = M_{\kk}(\psi_n + d_n)$. This, in particular, means that $M_0(f)\geq M_0(\psi_n)$, and hence the parameter $\beta$ (see Lemma~\ref{l::splinesFamily} and its proof) of the spline $\psi_n$ can be made larger (or kept the same) so that  for the obtained spline $\psi^*_n$ we have $M_{\kUpper}(f) = M_{\kUpper}(\psi^*_n)$ (hence $\psi^*_n\in X$), $M_0(\psi^*_n) = \delta$ and $M_k(\psi^*_n)\geq M_k(\psi_n) = M_k(\psi_n + d_n)$. Thus 
\begin{equation}\label{extremalSeqOfSplines}
       \lim_{n\to\infty} M_k(\psi^*_n)= \omega(X;\delta).
    \end{equation}

    Let $\psi^*_n = \alpha_n\psi_r(a_n,b_n,c_n) \in S_{\kk}$ for all $n\in\NN$. It is now sufficient to show that from each of the sequences $\{\alpha_n\}$, $\{a_n\}$, $\{b_n\}$ and $\{c_n\}$ one can extract a converging subsequence to numbers $\alpha, a,b,c$ respectively such that $\alpha, b > 0$.

    First we prove that the sequence $\alpha_n$ is separated from zero. Assume the contrary, switching to a subsequence if needed, we may suppose that $\lim_{n\to\infty} \alpha_n = 0$. However, from Kolmogorov's inequality~\eqref{kolmogorovInequality} we obtain that for some $C>0$
    $$
    M_k(\psi^*_n)\leq C\cdot \delta^{1-\frac kr}\cdot \alpha_n^{\frac kr}\to 0, n\to\infty,
    $$
    which contradicts to~\eqref{extremalSeqOfSplines}.

    Next we prove the sequence $\alpha_n$ is bounded from above. Again, assuming the contrary, we may suppose that $\alpha_n\to\infty$, $n\to\infty$. Then from the assumption on the the set $A$ we obtain that  one of the sequences $M_{r-1}(\psi^*_n)$ or $M_{r-2}(\psi^*_n)$ tends to zero as $n\to\infty$, and using the Kolmogorov inequality once again, we obtain that $M_k(\psi^*_n)\to 0$, $n\to\infty$, which is impossible.

    Thus we may assume that $\lim_{n\to\infty}\alpha_n = \alpha > 0$.

    Using similar arguments as above we obtain that $M_{r-1}(\psi^*_n)$ can not tend to $0$ as $n\to\infty$ (otherwise $M_{k}(\psi^*_n)\to 0$, $n\to\infty$, which is impossible); moreover, by Kolmogorov's inequality boundedness of $M_0(\psi^*_n)$ and $M_r(\psi^*_n)$ implies boundedness of $M_{r-1}(\psi^*_n) = \alpha_n\cdot b_n$. Hence switching to a subsequence if needed, we may assume that $b_n\to b> 0$, $n\to\infty$.

    Finally, the sequences $\{a_n\}$ and $\{c_n\}$  must also be bounded, since otherwise the norm $M_0(\psi^*_n)$ would be unbounded. Hence both these sequences have converging (to non-negative numbers) subsequences.
\end{proof}
\subsection{Applications of Theorem~\ref{th::main}}
Next we show that using the observation from Theorem~\ref{th::main}, we can obtain inequality~\eqref{dragomirInequality}. Taking into account inequality~\eqref{inequalityInTermsOfModulusOfContinuity}, it is sufficient to be able to compute the values $\omega(X, D^1;\delta)$, $\delta > 0$ on the class $X = \{f\in L^3_{\infty,\infty}\colon p_\eta(f)\leq 1\}$, where $p_\eta(f) = \|f''\|^{1-\eta}\cdot \|f'''\|^{\eta}$, $\eta\in (0,1)$. In the notations of Theorem~\ref{th::main}, this class is defined by the set $A = \{(x,y)\colon x,y>0, x^{1-\eta} y^\eta \leq 1\}$, which satisfies the assumptions of Theorem~\ref{th::main}.

Since for $\lambda >1$ and  $f\in L^3_{\infty,\infty}$, one has  $g(t) :=  f(\lambda t)\in L^3_{\infty,\infty}$, $M_0(f) = M_0(g)$, $M_k(g) > M_k(f)$ and $p_\eta(g) = \lambda^{2+3\eta} p_\eta(f)$, the supremum in definition~\eqref{modulusOfContinuity} can be taken over functions $f$ with $p_{\eta}(f) = 1$. 

Taking into account~\eqref{psi2}, 
$$
\|\psi_3(a,b,0)\| = \left|\int_0^{a+b}\psi_2(a,b,0;t)dt\right| = \frac{a^2b}{2} + ab^2 + \frac{b^3}{3}. 
$$
Thus in the case $\kk = (0,1,2,3)$, using~\eqref{splineNorms}, in order to compute the value $\omega(D^1;X;\delta)$, $\delta >0$, we arrive at the following extremal problem 
$$
\begin{cases}
    \alpha \left(ab + \frac{b^2}{2}\right)\to \sup,\\
    \alpha\left(\frac{a^2b}{2} + ab^2 + \frac{b^3}{3}\right) = \delta,\\
    \alpha b^{1-\eta} = 1,
\end{cases}
$$
for parameters $\alpha,a,b$. Thus the sharp constant in~\eqref{dragomirInequality} can be found by solving this extremal problem, which is a standard computational task.

In the general case of $\kk\in K$ (with  arbitrary $0<k<r$), the modulus of continuity $\omega$ can be computed from a similar extremal problem. Expressions for the norms  of special classes of splines (which include the spaces $S_{\kk}$, $\kk\in K$) were obtained in~\cite[Lemma~2]{Dzyadyk74}.

\bibliographystyle{elsarticle-num}
\bibliography{bibliography}
\end{document}